\documentclass[preprint,12pt]{elsarticle}
\usepackage{amsmath}
\usepackage{amssymb}
\usepackage{amsthm}
\usepackage{indentfirst}
\usepackage[latin1]{inputenc}
\usepackage{geometry}
\usepackage{amsfonts}
\usepackage{enumitem}
\usepackage[bookmarks=false,colorlinks=true,linkcolor=black,citecolor=black,filecolor=black,urlcolor=black]{hyperref} 
\usepackage{a4wide}
\usepackage[english]{babel}
\usepackage{longtable}
\usepackage{url}
\usepackage[all]{xy}
\usepackage{nicefrac}
\usepackage{amsfonts}
\usepackage{verbatim}

\usepackage{tikz}
\usetikzlibrary{matrix}
\usetikzlibrary{arrows}
\usetikzlibrary{fit}
\usetikzlibrary{shapes}

\usepackage{booktabs}

\DeclareSymbolFont{largesymbols}{OMX}{yhex}{m}{n}
\DeclareMathAccent{\widehat}{\mathord}{largesymbols}{"62}

\newcommand{\Q}{{\mathbb Q}}
\newcommand{\Z}{{\mathbb Z}}

\newcommand{\N}{{\mathbb N}}
\newcommand{\R}{{\mathbb R}}
\newcommand{\F}{\mathbb{F}}
\renewcommand{\H}{{\mathbb H}}
\newcommand{\U}{\mathcal{U}}
\renewcommand{\O}{\mathcal{O}}

\newcommand{\GL}{\textnormal{GL}}
\newcommand{\SL}{\textnormal{SL}}
\newcommand{\PSL}{\textnormal{PSL}}

\newcommand{\Bass}[3]{u_{#1,#2}(#3)}
\newcommand{\suma}[1]{\widehat{#1}}

\newcommand{\inv}{{^{-1}}}
\newcommand{\GEN}[1]{\left\langle #1 \right\rangle}
\newcommand{\BA}{{\mathcal{B}_1}}
\newcommand{\BI}{{\mathcal{B}_2}}
\newcommand{\mat}[4]{\left(\begin{array}{cc} #1 & #2 \\ #3 & #4 \end{array}\right)}
\newcommand{\iso}{\cong}

\newcommand{\id}{{\operatorname{id}}}

\newcommand{\quat}[2]{\left(\frac{#1}{#2}\right)}
\newcommand{\Ram}{\textnormal{Ram}}
\newcommand{\Jac}{\textnormal{Jac}}

\theoremstyle{plain}
\newtheorem{theorem}{Theorem}[section]
\newtheorem{definition}[theorem]{Definition}
\newtheorem{method}[theorem]{Method}

\newtheorem{corollary}[theorem]{Corollary}
\newtheorem{proposition}[theorem]{Proposition}
\theoremstyle{definition}
\newtheorem{remark}[theorem]{Remark}

\makeatletter
\def\step{%
   \@ifnextchar[ \@step{\@noitemargtrue\@step[\@itemlabel]}}
\def\@step[#1]{\item[#1]\mbox{}\\\hspace*{\dimexpr-\labelwidth-\labelsep}}
\makeatother

\journal{Journal of Pure and Applied Algebra}

\begin{document}

\begin{frontmatter}

\title{Describing units of integral group rings up to commensurability \tnoteref{label1}}

\author{Florian Eisele}
\ead{feisele@vub.ac.be}

\author{Ann Kiefer}
\ead{akiefer@vub.ac.be}

\author{Inneke Van Gelder}
\ead{ivgelder@vub.ac.be}

\address{Department of Mathematics, Vrije Universiteit Brussel, Pleinlaan 2, 1050 Brussels, Belgium}

\tnotetext[label1]{The research is supported by the Research Foundation Flanders (FWO - Vlaanderen), partially by FWO  project G.0157.12N}

\begin{abstract}
We restrict the types of $2\times2$-matrix rings which can occur as simple components in the Wedderburn decomposition of the rational group algebra of a finite group. This results in a description up to commensurability of the group of units of the integral group ring $\Z G$ for all finite groups $G$ that do not have a non-commutative Frobenius complement as a quotient.
\end{abstract}

\begin{keyword}
Units in integral group rings \sep Rational group representations \sep Wedderburn components

\MSC[2010] 16S34 \sep 16U60 \sep 20C05
\end{keyword}

\end{frontmatter}

\section{Introduction}

For a group $G$ we denote by $\U(\Z G)$ the unit group of the integral group ring $\Z G$. Bass \cite{bass1966} proved that if $C$ is a finite cyclic group, then the so-called Bass cyclic units generate a subgroup of finite index in $\U(\Z C)$. Next Bass and Milnor proved this result for finite abelian groups. This work stimulated the search for subgroups of finite index in the unit group of the integral group ring of finite non-abelian groups. For several classes of finite groups $G$ including nilpotent groups of odd order, Ritter and Sehgal \cite{RS1989,1991RitterSehgal,RS1991} showed that the Bass units together with the bicyclic units generate a subgroup of finite index in $\U(\Z G)$. 

Later on, Jespers and Leal \cite{JespersLeal1993} proved that the group generated by the Bass units and the bicyclic units is of finite index in $\U(\Z G)$
for finite groups $G$ which do not have any non-abelian homomorphic images that are  Frobenius complements (equivalently, are fixed point free, see Definitions \ref{fixed} and \ref{frobenius}) and, additionally, whose rational group ring $\Q G$ does not have any simple components of one of the following types:
\begin{enumerate}
\item a $2\times 2$-matrix ring over the rationals;
\item a $2\times 2$-matrix ring over a quadratic imaginary extension of the rationals;
\item a $2\times 2$-matrix ring over a non-commutative division algebra.
\end{enumerate}
The reason why some simple components are excluded is that not all congruence subgroups in the corresponding special linear groups over a maximal order are generated up to finite index by elementary matrices. We will explain this in Section \ref{preliminaries}. 

In this paper we give a classification of  all isomorphism types of $2\times 2$-matrix rings which actually occur as exceptional components (see Definition \ref{exceptional}), and all finite groups which have such an exceptional component as a faithful Wedderburn component. 
This is related to the heavy classification results of Banieqbal \cite{Banieqbal} and Nebe \cite{NebeQuat}, on which we opted to rely as little as possible. 
For the most part we use elementary techniques.
Only the proof of Theorem \ref{classification} relies on \cite{NebeQuat}, and \cite{Banieqbal} is not used at all. 
It should be noted that, mostly due to the use of the classification \cite{NebeQuat} instead of \cite{Banieqbal}, Theorem \ref{classification} can serve to significantly shorten the part in the work of Caicedo and del R\'io (\cite{2013Caicedo}) that deals with Banieqbal's classification \cite{Banieqbal}.

In Theorems \ref{main1} and \ref{main2} we prove that if $\Q G$ has a simple component which is a $2\times 2$-matrix ring over a quadratic imaginary extension of the rationals or a $2\times 2$-matrix ring over a totally definite quaternion algebra over $\Q$, then this component is a $2\times 2$-matrix ring over one of the following rings:
\begin{enumerate}
 \item \label{type1} $\Q(\sqrt{-1})$, $\Q(\sqrt{-2})$ or $\Q(\sqrt{-3})$;
\item \label{type2}$\quat{-1, -1}{\Q}$, $\quat{-1, -3}{\Q}$ or $\quat{-2, -5}{\Q}$.
\end{enumerate}
All of these fields respectively skew-fields contain a norm Euclidean maximal order (see  \cite[Proposition 6.4.1]{1963Weiss} and \cite[Theorem 2.1]{2012Fitzgerald}). This makes it possible to describe generators for a subgroup of finite index in $\SL_2(\O)$, with $\O$ the maximal order in one of the algebras listed under points \ref{type1} and \ref{type2} (see Proposition \ref{prop_generators}). Moreover, we are able to describe an effective method to describe generators for a subgroup $S$ of $\U(\Q G)$ which is commensurable with the group of units of $\Z G$ for a finite group $G$ (i.e. $S\cap\U(\Z G)$ has finite index in both $S$ and $\U(\Z G)$) provided that $G$ has no epimorphism onto a non-commutative Frobenius complement. This may be seen as a generalization of the result of Jespers and Leal \cite{JespersLeal1993} we mentioned earlier, which gives generators for a subgroup of finite index in $\U(\Z G)$ under stronger conditions on $G$.


Other constructions of subgroups of finite index of $\U(\Z G)$ for some specific classes of finite groups $G$, using a description of the center $\mathcal{Z}(\U(\Z G))$ and a description of the matrix units in $\Q G$, have been given for example in \cite{Jespers2010,2013JdROVG}.

\section{Preliminaries}\label{preliminaries}

We start by introducing some notations and definitions.

\begin{definition}
	Given a field $K$ as well as two elements $a,b\in K$ we define 
	the quaternion algebra $\quat{a, b}{K}$ as follows:
	\begin{equation*}
		\quat{a,b}{K} = \frac{K\langle i, j\rangle}{(i^2=a,\ j^2=b,\ ij=-ji)}.
	\end{equation*}
\end{definition}

The following classical theorem classifies quaternion algebras over number fields up to isomorphism, see \cite[Chapter 3, Th\'eor\`eme 3.1]{1980Vigneras} for a proof of this statement or the more general Hasse-Brauer-Noether-Albert Theorem in \cite[Theorem 32.11]{Reiner1975}.
\begin{theorem}\label{classification_quat}
Let $K$ be a number field. If $D$ is a quaternion algebra over $K$, the set $\Ram(D)\subset S(K)$ of places $v$ such that $D$ is ramified at $v$, i.e. such that $D\otimes_K K_v$ is not split, is a finite set with an even number of elements. Moreover, for any finite set $S\subset S(K)$ such that $|S|$ is even, there is a unique quaternion algebra with center $K$ such that $\Ram(D) = S$.
\end{theorem}

It is well known that for each finite group $G$ there exists a splitting $p$-modular system in the sense of the following definition. 
\begin{definition}
If $p$ is a prime, $(E,R,F)$ is called a $p$-modular system if $R$ is a complete discrete valuation ring of characteristic zero with maximal ideal $R\pi$ such that its residue field $F=R/R\pi$ has characteristic $p$ and $E$ is the field of fractions of $R$. It is called a $p$-modular splitting system for $G$ if $E$ and $F$ are both splitting fields for $G$.
\end{definition}

We recall the definitions of fixed point free groups and Frobenius complements.
\begin{definition}\label{fixed}
 A finite group $F$ is said to be fixed point free if it has an (irreducible) complex representation $\rho$ such that $1$ is not an eigenvalue of $\rho(f)$ for all $1\neq f\in F$.
\end{definition}

\begin{definition}\label{frobenius}
 A finite group $G$ is said to be a Frobenius group if it contains a proper non-trivial subgroup $H$ such that $H \cap H^g = \{1\}$ for all $g\in (G \setminus H)$. The group $H$ is called a Frobenius complement in $G$.
\end{definition}

Fixed point free groups are strongly related to Frobenius complements, i.e. a finite group G is fixed point free if and only if G has an epimorphism onto a Frobenius complement \cite[Theorem 18.1.v]{1968Passman}.

\begin{definition}\label{basicdef}
Let $G$ be a finite group. 
\begin{enumerate}
\item If $H\leq G$, we set $\suma{H} = \frac{1}{|H|}\sum_{h\in H}h$ an idempotent of $\Q G$. If $x\in G$ then we set $\suma{x}=\suma{\GEN{x}}$.
\item Let $g$ be an element of $G$ of order $n$ and $k$ and $m$ positive integers such that $k^m\equiv 1 \mod n$. Then
$$\Bass{k}{m}{g}=(1+g+\dots + g^{k-1})^{m}+\frac{1-k^m}{n}(1+g+\dots+g^{n-1})$$ is a unit of the integral group ring $\Z G$ called a \emph{Bass unit}.
\item The \emph{bicyclic units} of $\Z G$ are the elements of one of the following forms
$$ \beta_{g,h}= 1+(1-g)h(1+g+g^2+\dots+g^{n-1}),\ 
 \gamma_{g,h}= 1+(1+g+g^2+\dots+g^{n-1})h(1-g),
$$ 
where $g,h\in G$ and $g$ has order $n$. 
\item Let $\BA$ denote the group generated by the Bass units of $\Z G$ and $\BI$ the group generated by the bicyclic units of $\Z G$.
\item More generally, given a collection $\{e_1,...,e_n\}$ of idempotents of $\Q G$, we define generalized bicyclic units 
$$ \beta_{i,g}= 1+z^2(1-e_i)ge_i,\ 
 \gamma_{i,g}= 1+z^2e_ig(1-e_i),
$$ 
where $z\in \N$ is chosen of minimal absolute value with respect to the property that $ze_i$ lies in $\Z G$. We shall call the group generated by the various $\beta_{i,g}$ and $\gamma_{i,g}$ the group of \emph{generalized bicyclic units} and also denote it by $\BI$ (i.e., when we refer to $\BI$, a collection of idempotents should be given at least implicitly; for the ordinary bicyclic units this system of idempotents would be formed by the elements of the form $\suma{g}$ where $g$ ranges over all elements of $G$).
\end{enumerate}
\end{definition}

We recall some notions concerning the rational group algebra $\Q G$. Let $e_1,\dots,e_n$ be the primitive central idempotents of $\Q G$, then $$\Q G=\Q Ge_1\oplus \dots\oplus \Q G e_n,$$ where each $\Q G e_i$ is identified with the matrix ring $M_{n_i}(D_i)$ for some division algebra $D_i$. For every $i$, let $\O_i$ be an order in $D_i$. Then $M_{n_i}(\O_i)$ is an order in $\Q Ge_i$. Denote by $\GL_{n_i}(\O_i)$ the group of invertible matrices in $M_{n_i}(\O_i)$ and by $\SL_{n_i}(\O_i)$ its subgroup consisting of matrices of reduced norm $1$. 

The following proposition is a reformulation of the result of chapter 22 from \cite{Sehgal1993}.

\begin{proposition}\label{main}
Assume $G$ is a finite group and let $U\leq \U(\Z G)$ be a subgroup of the unit group of the integral group ring. Let $\Q G=\bigoplus_{i=1}^n \Q G e_i=\bigoplus_{i=1}^n M_{n_i}(D_i)$ be the Wedderburn decomposition of $\Q G$ and let $\O_i\subset D_i$ denote orders in each division ring.

Then $U$ is of finite index in $\U(\Z G)$ if and only if both of the following hold:
\begin{enumerate}
 \item The natural image of $U$ in $K_1(\Z G)$ is of finite index.
\item For each $i\in\{1,...,n\}$, the group $U$ contains a subgroup of finite index in $1\times \dots \times 1\times \SL_{n_i}(\O_i)\times 1\times \dots\times 1$, a multiplicative subgroup of $\bigoplus_{i=1}^n M_{n_i}(D_i)$. 
\end{enumerate}
\end{proposition}

For an ideal $Q_i$ of $\O_i$ we denote by $E(Q_i)$ the subgroup of $\SL_{n_i}(\O_i)$ generated by all so-called $Q_i$-elementary matrices, that is, $E(Q_i)=\langle I+ qE_{lm} \mid  q\in Q_i,\ 1\leq l,m\leq n_i,\ l\neq m,\ E_{lm} \textrm{ the $(l,m)$-matrix unit} \rangle$. We summarize the following theorems \cite[Theorem 21.1, Corollary 21.4]{Bass1964}, \cite[Theorem 2.4, Lemma 2.6]{Vaserstein1973}, \cite[Theorem 24]{Liehl1981} and \cite[Theorem]{Venkataramana1994}.

\begin{theorem}[Bass-Vaser{\v{s}}te{\u\i}n-Liehl-Venkataramana] \label{elementary}
Let $Q_i$ be an ideal in $\mathcal O_i$.
If $n_i\geq 3$ then $[\SL_{n_i}(\O_i):E(Q_i)]<\infty$. If $n_i=2$ and if $D_i$ is different from $\Q$, a quadratic imaginary extension of $\Q$ and a totally definite quaternion algebra with center $\Q$, then $[\SL_{2}(\O_{i}):E(Q_{i})]<\infty$.
\end{theorem}

Note that Kleinert \cite{2000Kleinert} proved that if $D_i$ is non-commutative then $\SL_1(\O_i)$ is finite if and only if $D_i$ is a totally definite quaternion algebra. 

Since the image of $\BA$ in $K_1(\Z G)$ is always of finite index \cite[Theorem 5]{bass1966}, one is interested in finding additional units such that the group generated by $\BA$ and these additional units satisfies condition 2 of Proposition \ref{main}. Often the bicyclic units can be used for this purpose, but our main focus will lie on groups $G$ for which this is not the case.

This leads us to the definition of exceptional components. Note that this is a weaker version of the definition used in \cite{JespersLeal1993}, since the result of Venkataramana stated in Theorem \ref{elementary} immediately leads to a strengthening of the results of \cite{JespersLeal1993}. 
\begin{definition}\label{exceptional}
 Let $G$ be a finite group. An exceptional component of $\Q G$ is a simple component of the following types:
\begin{enumerate}
 \item[type 1:] a non-commutative division ring other then a totally definite quaternion algebra,
 \item[type 2:] a $2\times 2$-matrix ring over the rationals, a quadratic imaginary extension of the rationals or over a totally definite quaternion algebra over $\Q$.
\end{enumerate}
\end{definition}

The following result is a reformulation of \cite[Theorem 3.3]{JespersLeal1993}.
\begin{proposition}\label{Jespers}
Let $G$ be a finite group such that $\Q G=\bigoplus_{i=1}^n M_{n_i}(D_i)$ has no exceptional components. Let $e_1,...,e_{{k}}$ be a set of idempotents in $\Q G$ such that for each irreducible representation $\Delta_i: \Q G \rightarrow M_{n_i}(D_i)$ with $n_i> 1$ there is some index $j(i)$ such that $\Delta_i(e_{j(i)}) \notin \{0, 1\}$ (i.e. $\Delta_i(e_{j(i)})$ is a non-central idempotent in $M_{n_i}(D_i)$).
Then the Bass cyclic units together with the generalized bicyclic units associated with $e_1,...,e_{{k}}$ generate a subgroup of finite index in $\U(\Z G)$. 
\end{proposition}
We will restrict the types of exceptional components that can actually occur in group rings, and we will show that this restriction results in a more general method of constructing a subgroup of $\U (\Q G)$ which is commensurable with $\U (\Z G)$.

\section{Main results}

In this section we restrict the type of $2\times 2$-matrices which can occur as simple components in the Wedderburn decomposition of $\Q G$ for finite groups $G$. We also give a classification of those finite groups which have a faithful exceptional $2\times 2$-matrix ring component. 

\begin{theorem}\label{main1}
If $\Q G$ has a simple component which is a $2\times 2$-matrix ring over a quadratic imaginary extension of the rationals, then this component is a $2\times 2$-matrix ring over one of the following fields:
\begin{enumerate}
\item \label{main1type2} $\Q(\sqrt{-1})$, 
\item \label{main1type3} $\Q(\sqrt{-2})$ or 
\item \label{main1type4 }$\Q(\sqrt{-3})$.
\end{enumerate}
Moreover, a finite subgroup $G$ of $\U(M_2(K))$, where $K$ is one of the above quadratic imaginary extension of the rationals, is solvable and $|G|=2^a 3^b$ for $a,b\in$ $\N$.
\end{theorem}
\begin{proof}
Let $K=\Q(\sqrt{-d})$ for a positive square-free integer $d$. Assume that there is a finite group $G$ and  a primitive central idempotent $e$ of $\Q G$ such that the simple component $\Q Ge$ is isomorphic to $M_2(K)$. We can assume without loss of generality that the $\Q$-representation of $G$ determined by $e$ is faithful and hence $G$ can be embedded into $M_2(K)$. We will identify $G$ with its image in $M_2(K)$.

Since the characteristic polynomial of a matrix in $M_2(K)$ always has degree 2 over $K$ and degree 4 over $\Q$, we conclude that the minimal polynomial of any element in $G$ over $\Q$ has degree 1, 2 or 4. Moreover, for any prime $p$, we know that the $\Q$-irreducible cyclotomic polynomial $\frac{X^p-1}{X-1}$ divides the minimal polynomial of $g\in G$ over $\Q$ whenever $g$ has order $p$. Therefore we deduce that if the order of an element of $G$ is a prime, then this prime must be equal to either $2$, $3$ or $5$, as these are the only primes whose associated cyclotomic polynomials have admissible degrees. Hence $|G|=2^x3^y5^z$ for some integers $x,y,z$.

We claim that $G$ has no elements of order 5.  Indeed, suppose $g \in G$ and $o(g)=5$. Then the minimal polynomial of $g$ over $K$, which we shall denote by $\mu_g$, divides $X^5-1$ and has degree at most 2. Moreover $\mu_g$ cannot be of degree $1$, because then $g$ would be a scalar matrix and a scalar matrix in $M_2(K)$ cannot have order $5$. 
Clearly  $X^5-1$ factorizes completely over the field $\Q(\zeta_5)$, where $\zeta_5$ is a $5$th root of unity. Indeed in $\Q(\zeta_5)$, 
$$X^5-1=\left(X-1\right)\left(X-\zeta_5 \right)\left(X-\zeta_5^2 \right)\left(X-\zeta_5^3 \right)\left(X-\zeta_5^4 \right).$$
Consequently $\mu_g$ is a product of two out of these five terms and we have two possibilities:
\begin{enumerate}
\item The constant term of $\mu_g$ is of the form $\zeta_5^{i}$ for some $1 \leq i \leq 4$. Thus $\zeta_5^{i} \in K$ and hence also $\zeta_5 \in K$. This implies that $\left[ K:\Q \right] \geq 4$, which contradicts the assumption that $\left[ K:\Q \right] = 2$. 
\item The constant term of $\mu_g$ is equal to one, and the coefficient of $X$ is equal to $-(\zeta_5+\zeta_5^{4}) \in \R - \Q$ or $-(\zeta_5^2+\zeta_5^{3}) \in \R - \Q$, which would imply that 
$K$ contains the totally real field $\Q[\zeta_5 + \zeta_5^{-1}]=\Q [\sqrt{5}]$, which is impossible since by assumption $K$ is a quadratic imaginary extension of $\Q$.
\end{enumerate}
Hence $|G|=2^x3^y$ for some integers $x,y$. By Burnside's $p^aq^b$-theorem, this implies that $G$ is solvable. This proves our claim on the solvability and possible orders of finite subgroups of $\U(M_2(K))$.

Assume that $g\in G$ has order $p^n$ with either $p=2$ and $n > 3$ or $p=3$ and $n>1$.  In this case the minimal polynomial of $g$ over $\Q$ is divisible by $\frac{X^{p^n}-1}{X^{p^{n-1}}-1}$, which has degree bigger than 4, a contradiction.

Now let us consider the case where $G$ has an element $g$ of order 8. Denote by $\mu_{\Q,g}$ and $\mu_{K,g}$ the minimal polynomials of $g$ over $\Q$ and $K$ respectively. Their degrees are 4 and 2, respectively. The field $\Q(\zeta_8)\cong \frac{\Q[X]}{(\mu_{\Q,g})}$ is an extension of $\Q$ of degree $4$ and the field $K(\zeta_8) \cong \frac{K[X]}{(\mu_{K,g})}$ is an extension of $K$ of degree $2$. However, since  $K/\Q$ is also a field extension of degree 2 and $\Q(\zeta_8)\subseteq K(\zeta_8)$, we have that $\Q(\zeta_8)=K(\zeta_8)$ and hence $K\subseteq \Q(\zeta_8)$. Using Galois' fundamental theorem we conclude that $K$ is either $\Q(\sqrt{-1})$, $\Q(\zeta_8+\zeta_8^{-1})$ or $\Q(\sqrt{-2})$. Since $K$ is assumed to be totally imaginary, we can exclude $\Q(\zeta_8+\zeta_8^{-1})$.

Assume from now on that $G$ has no element of order 8. Then the exponent of $G$ divides 12 and hence $K\subseteq \Q(\zeta_{12})$ ($\Q(\zeta_{12})$ is even a splitting field of $G$). Again using Galois' theorem we deduce that $K$ is either $\Q(\sqrt{-1})$, $\Q(\sqrt{-3})$ or $\Q(\sqrt{3})$. We can exclude $\Q(\sqrt{3})$ since this is a real field extension of $\Q$. This finishes the proof.
\end{proof}

\begin{remark}\label{rationals}
	Note that not only finite subgroups of $\U(M_2(K))$ for  quadratic imaginary field extensions $K$ but also finite subgroups of $\mathcal U(M_2(\Q))$ are solvable.
	This is clear since the only possible prime power orders for elements of finite order in $\U(M_2(\Q))$ are 1, 2, 3 and 4 (again by looking at the degrees of the corresponding cyclotomic polynomials, and noting that the degree of the minimal polynomial of an element is bounded by two).
\end{remark}

The following proposition turns the classification of all finite subgroups
of $\U(M_2(K))$ for $K\in\left\{\Q, \Q(\sqrt{-1}), \Q(\sqrt{-2}), \Q(\sqrt{-3})\right\}$ into a finite problem.
\begin{proposition}\label{finiteproblem}
	Let $K\in\left\{\Q, \Q(\sqrt{-1}), \Q(\sqrt{-2}), \Q(\sqrt{-3})\right\}$, and let $G\leq \U(M_2(K))$ be a finite group. 
	Then $G$ can be embedded in the finite group $\GL(2,25)$.
\end{proposition}
\begin{proof}
	The finite field $\F_{25}$ contains roots of $-1$, $-2$ and $-3$. By Hensel's lemma it follows that
	the given fields $K$ can all be embedded in $\Q_{5^2}$, the unramified extension of $\Q_5$ of degree 2. Let $\Z_{5^2}$ be the integral closure of $\Z_5$ in $\Q_{5^2}$, and let $\pi$ denote a generator of the maximal ideal of $\Z_{5^2}$ (so, in particular, $\Z_{5^2}/\pi \Z_{5^2} \iso \F_{25}$). The algebra $M_2(\Q_{5^2})$ contains the maximal order $M_2(\Z_{5^2})$, which is unique up to conjugation. Therefore we can assume without loss that $G \leq \U(M_2(\Z_{5^2}))$.
	There is some $n$ such that $G \cap (1+5^n\cdot M_2(\Z_{5^2})) = \{1\}$, and therefore we can again assume without loss that $G \leq \mathcal U(M_2(\Z_{5^2}/5^n\cdot \Z_{5^2}))$. Consider the short exact sequence
	\begin{equation}
		1 \longrightarrow (1+\pi\cdot M_2(\Z_{5^2})) / (1+5^n\cdot M_2(\Z_{5^2})) \longrightarrow \mathcal U(M_2(\Z_{5^2}/5^n\Z_{5^2})) \longrightarrow \GL(2,25) \longrightarrow 1. 
	\end{equation}
	The leftmost term in this sequence is a $5$-group, and we know by Theorem \ref{main1} that $2$ and $3$ are the only possible prime divisors of the order of $G$. Therefore $G$ intersects trivially with the kernel of the epimorphism to $\GL(2,25)$, which proves our claim.
\end{proof}

We will use the following proposition in the proof of
Theorem \ref{main2} below.

\begin{proposition}\label{prop_quot_quat}
	Let $D$ be a quaternion algebra with center $\Q$ and $M\subset D$ a maximal $\Z$-order. If $D$ is ramified at the prime $p$, then 
	$\Z_p\otimes M/\Jac(\Z_p\otimes M) \cong \F_{p^2}$. 
\end{proposition}
\begin{proof} This is an immediate consequence of \cite[Theorem 14.3]{Reiner1975}. 
\end{proof}


\begin{theorem}\label{main2}
Let $G$ be a finite group.
If $\Q G$ has a simple component which is a $2\times 2$-matrix ring over a totally definite quaternion algebra with center $\Q$, then this component is a $2\times 2$-matrix ring over one of the following algebras:
\begin{enumerate}
\item \label{main2type1}$\quat{-1, -1}{\Q}$,
\item \label{main2type2}$\quat{-1, -3}{\Q}$ or 
\item\label{main2type3}$\quat{-2, -5}{\Q}$.
\end{enumerate}
\end{theorem}
\begin{proof}
Let $D=\quat{a,b}{\Q}$ be a totally definite quaternion algebra over $\Q$. Assume that there is a finite group $G$ and a primitive central idempotent $e$ of $\Q G$ such that the simple component $\Q Ge$ is isomorphic to $M_2(D)$. We can assume without loss of generality that the $\Q$-representation of $G$ determined by $e$ is faithful and hence $G$ can be embedded into $M_2(D)$. We will identify $G$ with its image in $M_2(D)$.

Since the reduced characteristic polynomial of a matrix in $M_2(D)$ always has  degree 4 over $\Q$, we conclude that the minimal polynomial of any element in $G$ over $\Q$ has degree at most 4. Moreover, when $p$ is prime we know that the cyclotomic polynomial $\frac{X^p-1}{X-1}$ is irreducible over $\Q$ and it divides the minimal polynomial of $g\in G$ over $\Q$ whenever $g$ has order $p$. Therefore we deduce that $2$, $3$ and $5$ are the only primes that can occur as orders of elements of $G$. Hence $|G|=2^x3^y5^z$ for some integers $x,y,z$.

We claim that if $D$ is ramified at a prime $p$ then $p$ divides the order of $G$. Suppose by way of contradiction that $p$ is a prime such that $D$ ramifies at $p$ but $p$ does not divide the order of $G$. Let $(E,R,F)$ denote a splitting $p$-modular system for $G$. Then by \cite[Theorem (41.1)]{Reiner1975} the $R$-order $RG$ is maximal. Hence $RGe$ is a maximal order in $EGe \iso M_4(E)$, which implies $RGe\iso M_4(R)$. Let $M$ be a maximal order in $\Q Ge$ containing $\Z G e$. Then $M_4(R)\iso RGe\iso R\otimes \Z Ge \iso R\otimes M$, the last isomorphism being due to the fact that $R\otimes M$ would by construction be an $R$-order containing $R \otimes \Z G e$, which is already known to be maximal. Since $D$ is ramified at $p$ the $\Q_p$-algebra $\Q_p\otimes \Q Ge$ is not split. This implies in particular that $\Z_p\otimes M$ does not contain a system of more than two non-trivial orthogonal idempotents and hence neither does $\F_p\otimes M$  since idempotents of $\F_p\otimes M$ can be lifted to $\Z_p\otimes M$ (see \cite[Theorem 6.7]{1981CurtisReiner}). It follows that if $\F_p \otimes M$ is semisimple, then it must be isomorphic to either $\F_{p^{16}}$, $\F_{p^8}\oplus \F_{p^8}$ or $M_2(\F_{p^4})$, since all other semisimple $\F_p$-algebras of dimension $16$ have a larger system of orthogonal primitive idempotents. If $\F_p \otimes M$ is not semisimple, then its Jacobson radical is non-zero. This would imply that the Jacobson radical of $F\otimes M$ is non-zero as well. In any case it follows that $F\otimes M$ is not isomorphic to $M_4(F)$ and hence $R\otimes M$ is not isomorphic to $M_4(R)$, a contradiction.

Hence the only finite primes $p$ where $D$ can be ramified are $p=2,3,5$. Theorem \ref{classification_quat} now implies that the only possible sets of ramified places for $D$ are $\left\{ 
\infty, 2 \right\}$, $\left\{ \infty, 3 \right\}$, $\left\{ \infty, 5 \right\}$ and $\left\{ \infty, 2, 3, 5 \right\}$. The following 
quaternion algebras have these respective sets as their set of ramified places: $\quat{-1, -1}{\Q}$ (places $\infty$, 2), $\quat{-1, -3}{\Q}$ 
(places $\infty$, 3), $\quat{-2, -5}{\Q}$ (places $\infty$, 5) and $\quat{-3,-10}{\Q}$ (places $\infty$, 2, 3, 5). By the 
uniqueness asserted in Theorem \ref{classification_quat}, it follows that $D$ is isomorphic to one of these four quaternion algebras.  It remains to prove that the case $D\cong \quat{-3,-10}{\Q}$ cannot occur.

Assume by way of contradiction that $D\iso \quat{-3,-10}{\Q}$. The algebra $D$ is ramified at the primes $3$ and $5$ and hence $G$ contains an element of order $3$ and an element of order $5$. We will use this to show that $G$ must be non-solvable.
To see this we first bound the $3$-part and the $5$-part of the order of $G$. Note that $G$ embeds into $\mathcal U(M_p/p^iM_p)$ for any prime $p$ and sufficiently large $i$, where $M$ is a maximal order in $\Q Ge$ containing $\Z G e$ and $M_p := \Z_p\otimes M$ denotes the completion of $M$ at the prime $p$. The finite group $\mathcal U(M_p/p^iM_p)$ is an extension of $\mathcal U(M_p/\Jac(M_p))$ by the  $p$-group $1+\Jac(M_p) / 1+p^i M_p$.
As a consequence, every $p'$-subgroup of $G$ can be embedded in $\mathcal U(M_p/\Jac(M_p))$. In particular, a Sylow $3$-subgroup of $G$ can be embedded in 
$\mathcal U(M_5/\Jac(M_5))\iso \GL(2,25)$ (by Proposition \ref{prop_quot_quat}) and a Sylow $5$-subgroup of $G$ can be embedded in $\mathcal U(M_3/\Jac(M_3))\iso \GL(2,9)$. It follows that
the Sylow $5$-subgroup of $G$ has order at most $5$, and the Sylow $3$-subgroup of $G$ has order at most $9$. But all groups of order $3\cdot 5$ and $3^2\cdot 5$ are abelian, and hence a Hall $\{3,5\}$-subgroup of $G$ (if it exists) must be abelian. In particular, a Hall $\{3,5\}$-subgroup of $G$ would contain an element of order $15$. If $G$ is solvable, then, by Hall's Theorem, all possible Hall-subgroups exist, which implies  that $G$ contains an element of order $15$. But an element of order 15 in $\mathcal U (M_2(D))$ would have minimal polynomial over $\Q$ divisible by the 
irreducible cyclotomic polynomial $\frac{X^{15}-1}{(X^5-1)(1+X+X^2)}$. This is impossible, since an element of $M_2(D)$ has minimal polynomial of degree at most $4$. We conclude that $G$ is non-solvable. 

Now we will place restrictions on the structure of the Sylow $2$-subgroups of $G$.
Let $S_2$ denote such a Sylow $2$-subgroup. 
Clearly, the $\Q$-span of $S_2$ in $M_2(D)$, which we shall denote by $A$, is a semisimple $\Q$-subalgebra of $M_2(D)$. As we have already shown that any finite group spanning $M_2(D)$ must be non-solvable, it is also clear that $A$ is properly contained in $M_2(D)$. Hence $\dim_{\Q}(A) < 16$. We get the following cases:
\begin{enumerate}
	\item $A$ is not simple: In this case 
	$A=D_1\oplus D_2$, and both $D_1$ and $D_2$ can be embedded in $D$. Note that each $D_i$ is spanned by units of $2$-power order, and the elements $1$ and $-1$ are the only elements of $2$-power order in $D$
	(since $\Q_5 \supset \Q(\sqrt{-1})$ does not split $D$). It follows that $D_1=D_2=\Q$, and therefore $S_2\iso V_4$.
	\item $A$ is simple and neither a skew-field nor a field: We get 
	$A=M_2(K)$, where $K$ is a field contained in $D$.
	This means that either $K=\Q$ or $K$ is a quadratic imaginary extension of $\Q$. 
	In the latter case Theorem \ref{main1} tells us that $K=\Q (\sqrt{-d})$ for $d\in\{1,2,3\}$. Since $A$ occurs as a simple component of $\Q S_2$, $Z(A)=K$ occurs as a simple component of $Z(\Q S_2)$, and we know that those are contained in $\Q(\zeta_{\exp(S_2)})$. But $\Q(\sqrt{-3})=\Q(\zeta_{3})$ cannot be embedded in $\Q(\zeta_{2^i})$ for any $i$, so $d=3$ is impossible. Moreover, $\Q(\sqrt{-1}) \subset \Q_5$ and $\Q(\sqrt{-2}) \subset \Q_3$. If $K$ were either one of those, then it would be a maximal subfield of $D$ and therefore a splitting field. This would imply that either  $\Q_3$ or $\Q_5$ is also a splitting field for $D$, but we know that this is not the case. The only possibility is hence $K=\Q$, and in this case $S_2\leq \GL_2(\R)$. Using some basic linear algebra (see for example \cite{Shafarevich}), one can see that finite subgroups of $\GL_2(\R)$ can only be cyclic or dihedral. Hence, by Remark \ref{rationals}, we get that $S_2\iso D_8$.
	\item $A$ is a skew-field or a field: In this case $S_2$ has a faithful fixed point free representation and is therefore a Frobenius complement. By \cite[Theorem 18.1]{1968Passman}, the Sylow $2$-subgroups of Frobenius complements are cyclic groups or generalized quaternion groups. So $S_2 \iso Q_{2^i}$ for some $i \geq 3$ or $S_2\iso C_{2^i}$ for some $i\geq 0$. Among those possibilities the quaternion groups $Q_8$ and $Q_{16}$ and the cyclic groups 
		of order $1$, $2$, $4$, $8$ and $16$ are the only ones whose group ring over $\Q$ has a faithful Wedderburn component of dimension strictly smaller than $16$.
\end{enumerate}
In particular we can conclude that the order of $S_2$ is at most $16$. 

For sufficiently large $i$ we can embed $G$ in $\U(M/2^i M)$, and
we have a short exact sequence
\begin{equation}
	1 \longrightarrow 1 + \Jac(M_2/2^iM_2) \longrightarrow \U(M_2/2^iM_2) \longrightarrow \U(M_2/\Jac(M_2)) \iso \GL(2,4) \longrightarrow 1
\end{equation}
where we again use the notation $M_2 := \Z_2 \otimes M$.
Since the leftmost group in this sequence is a $2$-group, and $G$ is non-solvable, the rightmost arrow must map the image of $G$ in $\U(M_2/2^iM_2)$ to a non-solvable subgroup of $\GL(2,4)$, which we shall denote by $\bar G$.
We know that $\SL(2,4)\iso \PSL(2,4)\iso A_5$ is a normal simple subgroup of $\GL(2,4)$ of index 3.
The group $\bar G \cap \SL(2,4)$ is a normal subgroup of $\bar G$, and $\bar G/\bar G \cap \SL(2,4) \hookrightarrow \GL(2,4)/\SL(2,4) \iso C_3$. If $\bar G \cap \SL(2,4)$ were properly contained in $\SL(2,4)$, then it would have to be solvable (as any group of order $<60$ is solvable). Then $\bar G$ would be solvable, which is impossible. It follows that $\SL(2,4) \subseteq \bar G$. Since $\SL(2,4)$ is a maximal subgroup of $\GL(2,4)$ it follows that $\bar G \in \{\SL(2,4), \GL(2,4)\}$. We already saw that $G$ may not contain an element of order $15$, and hence neither may $\bar G$. Thus $\bar G$ cannot be equal to $\GL(2,4)$, and we conclude $\bar G = \SL(2,4)$.

We have bounded the order of the Sylow $2$-subgroup of $G$ by $16$, and therefore we get a short exact sequence
\begin{equation}\label{eqn_XGSL}
	1 \longrightarrow X \longrightarrow G \longrightarrow \SL(2,4) \longrightarrow 1
\end{equation}
where $X$ is a $2$-group of size $\leq 4$. The automorphism groups of $2$-groups of those sizes are solvable, and therefore $\SL(2,4)$ has no non-trivial homomorphisms into those automorphism groups. Hence \eqref{eqn_XGSL} is a central extension. We know that the center of $G$ is contained in $Z(M_2(D))\iso \Q$, which implies that $X$ can have order at most $2$. 
Hence $G$ corresponds with an element in $H^2(\SL(2,4),C_2)$. It is well known that this cohomology group is of order 2 and thus $G$ is isomorphic with either $C_2\times \SL(2,4)$ (corresponding to the trivial element in the cohomology group) or with $\SL(2,5)$. If $G\iso C_2\times \SL(2,4)$, then $M_2(D)$ is a simple component of $\Q \SL(2,4)$. Hence we just have to check whether either $\Q\SL(2,4)$ or $\Q \SL(2,5)$ has $M_2(D)$ as a simple component. But the Wedderburn 
decompositions of $\Q\SL(2,4)$ and $\Q\SL(2,5)$ are known (alternatively, can be computed using \textsc{Gap}), and indeed neither $\Q\SL(2,4)$ nor $\Q \SL(2,5)$ has $M_2(D)$ as a simple component, which means that we have reached a contradiction.
\end{proof}

Theorems \ref{main1} and \ref{main2} and some computations done in \cite[Chapter 6.4]{1963Weiss} and \cite{2012Fitzgerald} yield the existence of norm Euclidean orders in the corresponding $2\times 2$-matrix ring components of rational group algebras. 
It is easy to verify that this norm Euclidean order has to be a (and hence the unique) maximal order $\O$. For quaternion algebras, it can even be shown that the listed algebras are the only possible positive definite quaternion algebras over $\Q$ having this property \cite[Theorem 2.1]{2012Fitzgerald}.

\begin{corollary}\label{orders}
	Let $G$ be a finite group with the property that $\Q G$ has a simple component $M_2(D)$ of one of the following types:
	\begin{enumerate}
		\item a $2\times 2$-matrix ring over $\Q$;
		\item a $2\times 2$-matrix ring over an imaginary quadratic
			extension of $\Q$;
		\item a $2\times 2$-matrix ring over a totally definite 
			quaternion algebra with center $\Q$.
	\end{enumerate}
	Then a maximal order of $D$ is norm Euclidean and therefore unique. The corresponding maximal orders are listed in Table \ref{orders_table} below.
\begin{longtable}{@{}ll@{}} \caption{Maximal orders} \label{orders_table} \\ \toprule[1.5pt]
Division ring & Maximal order \\ \midrule 
\endfirsthead \toprule[1.5pt] Division ring & Maximal order  \\ \midrule
\endhead \hline \multicolumn{2}{c}{continued}\\ \midrule[1.5pt]\endfoot\bottomrule[1.5pt]\endlastfoot
$\Q$ & $\Z$  \\
$\Q(\sqrt{-1})$ & $\Z[\sqrt{-1}]$  \\
$\Q(\sqrt{-2})$ & $\Z[\sqrt{-2}]$  \\
$\Q(\sqrt{-3})$ & $\Z\left[\frac{1+\sqrt{-3}}{2}\right]$  \\
$\quat{-1,-1}{\Q}$ & $\Z[1,i,j,\frac{1}{2}(1+i+j+ij)]$ \\
$\quat{-1,-3}{\Q}$ & $\Z[1,i,\frac{1}{2}(1+j),\frac{1}{2}(i+ij)]$ \\
$\quat{-2,-5}{\Q}$ & $\Z[1,\frac{1}{4}(2+i-ij),\frac{1}{4}(2+3i+ij),\frac{1}{2}(1+i+j)]$ \\
\end{longtable}
\end{corollary}

\begin{theorem}\label{classification}
	Let $G$ be a finite group with the property that $\Q G$ has a faithful
	simple component of one of the following types:
	\begin{enumerate}
		\item a $2\times 2$-matrix ring over $\Q$;
		\item a $2\times 2$-matrix ring over an imaginary quadratic
			extension of $\Q$;
		\item a $2\times 2$-matrix ring over a totally definite 
			quaternion algebra with center $\Q$.
	\end{enumerate}
	Then $G$ is isomorphic to one of the 55 groups listed in Table \ref{exceptional_table} below.
\end{theorem}
\begin{proof}
	The maximal finite subgroups of $2\times 2$-matrices over 
	totally definite quaternion algebras with center $\Q$ were classified in
	\cite{NebeQuat}  by means of investigating invariant rational lattices. Among those groups, the ones acting primitively are
	listed in a database which forms part of the computer algebra system \textsc{Magma}. The two groups acting imprimitively are
	explicitly given. By Proposition \ref{finiteproblem}, all finite subgroups of $M_2(\Q)$
	and $M_2(\Q(\sqrt{-d}))$ for $d\in \{1,2,3\}$ can be embedded in
	the finite group $\GL(2,25)$. The rest is a computation, which was performed using \textsc{Gap}. 
\end{proof}

Note that the problem of determining all groups in Theorem \ref{classification} was even a priori (i.e. without using \cite{NebeQuat}) clearly a finite problem: If $G\leq M_2(D)$ with $D$ a totally definite quaternion algebra over $\Q$, then it can be shown with techniques as in the proof op Proposition \ref{finiteproblem} that $G\leq \U(M_2(\O/7\O))$, where $\O$ is the maximal order of $D$. 

\begin{remark}[Notation in Table \ref{exceptional_table}]
	We use the abbreviations
	\begin{equation}
		\mathcal \H_1 = \quat{-1,-1}{\Q} \quad
		\mathcal \H_3 = \quat{-1,-3}{\Q} \quad
		\mathcal \H_5 = \quat{-2,-5}{\Q}
	\end{equation}
We use the standard short \textsc{Gap} notation for the group structure.
	In this description, cyclic groups are
	denoted by their order, e. g. ``$3$'' denotes the cyclic group $C_3$,
	and ``$3^2$'' denotes the elementary abelian group $C_3\times C_3$.
Direct products are denoted by $A\times B$ and semidirect products by $A:B$.
For non-split extension of $A$ by $B$, we write $A.B$.
\end{remark}

{\small 
\begin{longtable}{@{}lll@{}} \caption{Groups with faithful exceptional components} \label{exceptional_table} \\ \toprule[1.5pt]
\textsc{SmallGroup} ID & Structure & Faithful exceptional components  \\ \midrule 
\endfirsthead \toprule[1.5pt] \textsc{SmallGroup} ID & Structure & Faithful exceptional components  \\ \midrule 
\endhead \hline \multicolumn{3}{c}{continued}\\ \midrule[1.5pt]\endfoot\bottomrule[1.5pt]\endlastfoot
{[}6, 1{]} & $ S_3 $ & 1 $\times$  $M_2(\mathbb Q)$  \\
{[}8, 3{]} & $ D_8 $ & 1 $\times$  $M_2(\mathbb Q)$  \\
{[}12, 4{]} & $ D_{12} $ & 1 $\times$  $M_2(\mathbb Q)$  \\
{[}16, 6{]} & $8:2$ & 1 $\times$  $M_2(\mathbb Q (\sqrt{-1}))$  \\
{[}16, 
8
 {]} & $ {\rm QD}_{16} \textnormal{ (also denoted by } D_{16}^{-} \textnormal{)\
} $ & 1 $\times$  $M_2(\mathbb Q (\sqrt{-2}))$  \\
{[}16, 13{]} & $(4 \times 2):2$ & 1 $\times$  $M_2(\mathbb Q (\sqrt{-1}))$  \\
{[}18, 3{]} & $3 \times  S_3 $ & 1 $\times$  $M_2(\mathbb Q (\sqrt{-3}))$  \\
{[}24, 3{]} & $ {\rm SL} (2,3)$ & 1 $\times$  $M_2(\mathbb Q (\sqrt{-3}))$  \\
{[}24, 5{]} & $4 \times  S_3 $ & 1 $\times$  $M_2(\mathbb Q (\sqrt{-1}))$  \\
{[}24, 8{]} & $(6 \times 2):2$ & 1 $\times$  $M_2(\mathbb Q (\sqrt{-3}))$  \\
{[}24, 10{]} & $3 \times  D_8 $ & 1 $\times$  $M_2(\mathbb Q (\sqrt{-3}))$  \\
{[}24, 11{]} & $3 \times  Q_8 $ & 1 $\times$  $M_2(\mathbb Q (\sqrt{-3}))$  \\
{[}32, 8{]} & $2.((4 \times 2):2)=(2^2).(4 \times 2)$ & 1 $\times$  $M_2({\mathbb H}_{1})$  \\
{[}32, 11{]} & $(4 \times 4):2$ & 2 $\times$  $M_2(\mathbb Q (\sqrt{-1}))$  \\
{[}32, 44{]} & $(2 \times  Q_8 ):2$ & 1 $\times$  $M_2({\mathbb H}_{1})$  \\
{[}32, 50{]} & $(2 \times  Q_8 ):2$ & 1 $\times$  $M_2({\mathbb H}_{1})$  \\
{[}36, 6{]} & $3 \times (3:4)$ & 1 $\times$  $M_2(\mathbb Q (\sqrt{-3}))$  \\
{[}36, 12{]} & $6 \times  S_3 $ & 1 $\times$  $M_2(\mathbb Q (\sqrt{-3}))$  \\
{[}40, 3{]} & $5:8$ & 1 $\times$  $M_2({\mathbb H}_{5})$  \\
{[}48, 16{]} & $(3:8):2$ & 1 $\times$  $M_2({\mathbb H}_{1})$  \\
{[}48, 18{]} & $3: Q_{16} $ & 1 $\times$  $M_2({\mathbb H}_{3})$  \\
{[}48, 28{]} & $2. S_4 = {\rm SL} (2,3).2$ & 1 $\times$  $M_2({\mathbb H}_{3})$  \\
{[}48, 29{]} & $ {\rm GL} (2,3)$ & 1 $\times$  $M_2(\mathbb Q (\sqrt{-2}))$  \\
{[}48, 33{]} & $ {\rm SL} (2,3):2$ & 1 $\times$  $M_2(\mathbb Q (\sqrt{-1}))$  \\
{[}48, 39{]} & $(2 \times (3:4)):2$ & 1 $\times$  $M_2({\mathbb H}_{3})$  \\
{[}48, 40{]} & $ Q_8  \times  S_3 $ & 1 $\times$  $M_2({\mathbb H}_{1})$  \\
{[}64, 37{]} & $2.(((4 \times 2):2):2) = (4 \times 2).(4 \times 2)$ & 2 $\times$  $M_2({\mathbb H}_{1})$  \\
{[}64, 137{]} & $((4 \times 4):2):2$ & 2 $\times$  $M_2({\mathbb H}_{1})$  \\
{[}72, 19{]} & $(3^2):8$ & 2 $\times$  $M_2({\mathbb H}_{3})$  \\
{[}72, 20{]} & $(3:4) \times  S_3 $ & 1 $\times$  $M_2({\mathbb H}_{3})$  \\
{[}72, 22{]} & $(6 \times  S_3 ):2$ & 1 $\times$  $M_2({\mathbb H}_{3})$  \\
{[}72, 24{]} & $(3^2): Q_8 $ & 1 $\times$  $M_2({\mathbb H}_{3})$  \\
{[}72, 25{]} & $3 \times  {\rm SL} (2,3)$ & 3 $\times$  $M_2(\mathbb Q (\sqrt{-3}))$  \\
{[}72, 30{]} & $3 \times ((6 \times 2):2)$ & 2 $\times$  $M_2(\mathbb Q (\sqrt{-3}))$  \\
{[}96, 67{]} & $ {\rm SL} (2,3):4$ & 2 $\times$  $M_2(\mathbb Q (\sqrt{-1}))$  \\
{[}96, 190{]} & $(2 \times  {\rm SL} (2,3)):2$ & 1 $\times$  $M_2({\mathbb H}_{1})$  \\
{[}96, 191{]} & $(2. S_4 = {\rm SL} (2,3).2):2$ & 1 $\times$  $M_2({\mathbb H}_{1})$  \\
{[}96, 202{]} & $(2 \times  {\rm SL} (2,3)):2$ & 1 $\times$  $M_2({\mathbb H}_{1})$  \\
{[}120, 5{]} & $ {\rm SL} (2,5)$ & 1 $\times$  $M_2({\mathbb H}_{3})$  \\
{[}128, 937{]} & $( Q_8  \times  Q_8 ):2$ & 4 $\times$  $M_2({\mathbb H}_{1})$  \\
{[}144, 124{]} & $3:(2. S_4 = {\rm SL} (2,3).2)$ & 3 $\times$  $M_2({\mathbb H}_{3})$  \\
{[}144, 128{]} & $ S_3  \times  {\rm SL} (2,3)$ & 1 $\times$  $M_2({\mathbb H}_{1})$  \\
{[}144, 135{]} & $((3^2):8):2$ & 4 $\times$  $M_2({\mathbb H}_{3})$  \\
{[}144, 148{]} & $(2 \times ((3^2):4)):2$ & 2 $\times$  $M_2({\mathbb H}_{3})$  \\
{[}160, 199{]} & $((2 \times  Q_8 ):2):5$ & 1 $\times$  $M_2({\mathbb H}_{1})$  \\
{[}192, 989{]} & $((2. S_4 = {\rm SL} (2,3).2):2):2$ & 2 $\times$  $M_2({\mathbb H}_{1})$  \\
{[}240, 89{]} & $2. S_5  =  {\rm SL} (2,5).2$ & 1 $\times$  $M_2({\mathbb H}_{5})$  \\
{[}240, 90{]} & $ {\rm SL} (2,5):2$ & 1 $\times$  $M_2({\mathbb H}_{5})$  \\
{[}288, 389{]} & $((3:4) \times (3:4)):2$ & 2 $\times$  $M_2({\mathbb H}_{3})$  \\
{[}320, 1581{]} & $2.(((2^4):5):2) = (((2 \times  Q_8 ):2):5).2$ & 2 $\times$  $M_2({\mathbb H}_{1})$  \\
{[}384, 618{]} & $(( Q_8  \times  Q_8 ):3):2$ & 1 $\times$  $M_2({\mathbb H}_{1})$  \\
{[}384, 18130{]} & $(( Q_8  \times  Q_8 ):3):2$ & 1 $\times$  $M_2({\mathbb H}_{1})$  \\
{[}720, 409{]} & $ {\rm SL} (2,9)$ & 2 $\times$  $M_2({\mathbb H}_{3})$  \\
{[}1152, 155468{]} & $( {\rm SL} (2,3) \times  {\rm SL} (2,3)):2$ & 1 $\times$  $M_2({\mathbb H}_{1})$  \\
{[}1920, 241003{]} & $2.(2^4: A_5 )$ & 1 $\times$  $M_2({\mathbb H}_{1})$  \\
\end{longtable}}

\section{An effective method to compute $\U(\Z G)$ up to commensurability}\label{section:methods}

The proof of the following proposition is based on the Euclidean algorithm and is well known.
\begin{proposition}\label{prop_generators}
 Let $\O$ be a norm Euclidean order in either $\Q$, a totally definite quaternion algebra over $\Q$ or a quadratic imaginary extension of $\Q$.  Let $B$ be a $\Z$-basis of $\O$. Let $$X=\left\{  \mat{1}{x}{0}{1}, \mat{1}{0}{x}{1} : x\in B \right\}.$$ Then $X$ generates a subgroup of finite index in $\SL_2\left(\O\right)$.
\end{proposition}

Due to the restrictions we obtained on the possible exceptional components, together with Proposition \ref{prop_generators}, we can generalize Proposition \ref{Jespers}. That is, we can allow exceptional components of type 2. Recall that $\BA$ denotes the group generated by the Bass units of $\Z G$ and $\BI$ the group generated by the bicyclic units of $\Z G$.

\begin{method}\label{method}
Let $G$ be a finite group and let $\Q G=\bigoplus_{i=1}^n \Q G e_i=\bigoplus_{i=1}^n M_{n_i}(D_i)$ be the Wedderburn decomposition of $\Q G$. Assume that $\Q G$ has no simple component which is a non-commutative division ring other than a totally definite quaternion algebra. Assume furthermore that for each $i\in \{1,...,n\}$ with $n_i>2$ or $n_i=2$ and $D_i$ different from $\Q$, a quadratic imaginary extension of $\Q$ or a totally definite quaternion algebra over $\Q$, there exists a $g_i\in G\setminus\{1\}$ such that $\suma{g_i}e_i\notin \{0,e_i\}$.

For each component $\Q Ge_i\iso M_2(D_i)$, with $D_i$ one of the following rings:
\begin{enumerate}
 \item $\Q$;
 \item $\Q(\sqrt{-1})$, $\Q(\sqrt{-2})$ or $\Q(\sqrt{-3})$;
\item $\quat{-1, -1}{\Q}$, $\quat{-1, -3}{\Q}$ or $\quat{-2, -5}{\Q}$,
\end{enumerate}
compute the isomorphism $\iota_i:M_2(D_i)\stackrel{\sim}{\rightarrow} \Q Ge_i \subset \Q G$. Let $\O_i$ be a maximal order of $D_i$ with $\Z$-basis $B_i$ and set $$U_i:=\left\{  1+\iota_i\mat{0}{x}{0}{0}, 1+\iota_i\mat{0}{0}{x}{0} : x\in B_i\right\}.$$ 

This construction yields a subgroup $U:=\GEN{\BA\cup \BI \cup \bigcup_{i} U_i}$ of $\Q G$ which is commensurable with $\U(\Z G)$. 
\end{method}
\begin{proof}
Let $\O_i\subset D_i$ denote the maximal order in each division ring. It suffices to verify condition 2 of Proposition \ref{main} for each simple component of $\Q G$, i.e. $U\cap (1\times \dots \times 1\times \SL_{n_i}(\O_i)\times 1\times \dots\times 1)$ is of finite index in $1\times \dots \times 1\times \SL_{n_i}(\O_i)\times 1\times \dots\times 1$.

For the components with $n_i=1$, the result follows trivially since $D_i$ is either commutative or a totally definite quaternion algebra. In both cases $\SL_1(\O_i)$ is finite \cite{2000Kleinert}.

For the components $M_2(D)$ with $D$ equal to $\Q$, a quadratic imaginary extension of $\Q$ or a totally definite quaternion algebra over $\Q$, we know by Corollary \ref{orders} that a maximal order of $D$ is norm Euclidean and unique. The result then follows from Proposition \ref{prop_generators}. 

For the remaining components $\Q Ge_i$, there always exists a $g_i\in G$ such that $\suma{g_i}e_i$ is a non-central idempotent in $\Q Ge_i$.
By applying an adapted version of the proof of \cite[Corollary 4.1]{JespersLeal1993} one can conclude that this implies that $U$ 
contains a group of the form $1\times \dots \times 1\times E(Q_i)\times 1\times \dots\times 1$ for some ideal $Q_i \trianglelefteq \O_i$ (where  $E(Q_i)$ is defined as in Theorem \ref{elementary}). Now Theorem \ref{elementary} implies that condition 2 holds for all these components, which finishes the proof. 
\end{proof}

\begin{remark}\label{remark_Brown}
In Method \ref{method} we need a construction of non-central idempotents in $\Q Ge_i$ for two purposes:
\begin{itemize}
\item for the explicit determination of the isomorphisms $\iota_i:M_2(D_i)\stackrel{\sim}{\rightarrow} \Q Ge_i \subset \Q G$ when $\Q Ge_i$ is in the list $\{M_2(\Q),M_2(\Q(\sqrt{-1}), M_2(\Q(\sqrt{-2})), M_2(\Q(\sqrt{-3})),M_2(\quat{-1, -1}{\Q}),$ $ M_2(\quat{-1, -3}{\Q}), M_2(\quat{-2, -5}{\Q})\}$;
 \item for the use of bicyclic units or generalized bicyclic units (as in Definition \ref{basicdef}) when $\Q Ge_i$ is non-commutative and not in the list above.
\end{itemize}
The components where no $\suma{g}$ gives a non-central idempotent are classified. For each epimorphic image $G/N$ of $G$ which is isomorphic to a Frobenius complement, $G$ has an irreducible representation $\rho$ with kernel $N$ such that $\rho(\suma{g})$ acts as the identity for all $g\in G\setminus\{1\}$ \cite[Theorem 18.1.v]{1968Passman}. By \cite[Theorem 12.2]{2001Brown}, we even know that each such epimorphism determines a unique representation with this property and hence a unique simple component of $\Q G$ in which we might lack idempotents.

This means that Method \ref{method} applies to all finite groups $G$ that do not have a non-commutative Frobenius complement as a quotient (it follows easy that under this condition $\Q G$ has no non-commutative division ring as a simple component).
Although, if $G$ has an epimorphic image onto a Frobenius complement and if one knows other constructions of non-central idempotents, one can just work with those idempotents. One only has to modify Method \ref{method} to use the generalized bicyclic units based on those idempotents.
\end{remark}

\section{Examples}

Before we give some examples, we recall some definitions from \cite{Olivieri2004}.

If $K\lhd H\leq G$ then let $$\varepsilon(H,K)=\prod_{M/K\in\mathcal{M}(H/K)}
(\suma{K}-\suma{M})=\suma{K}\prod_{M/K\in\mathcal{M}(H/K)} (1-\suma{M}),$$ where $\mathcal{M}(H/K)$ denotes the set of all minimal normal subgroups of
$H/K$. We extend this notation by setting $\varepsilon(H,H)=\suma{H}$. Clearly $\varepsilon(H,K)$ is an idempotent of the group algebra $\Q G$. Let
$e(G,H,K)$ be the sum of the distinct $G$-conjugates of $\varepsilon(H,K)$, that is, if $T$ is a right transversal of $C_G(\varepsilon(H,K))$ in $G$,
then
    $$e(G,H,K)=\sum_{t\in T}\varepsilon(H,K)^t.$$
Clearly, $e(G,H,K)$ lies in the center of $\Q G$ and if the $G$-conjugates of $\varepsilon(H,K)$ are orthogonal, then $e(G,H,K)$ is a central idempotent of $\Q G$.

\subsection{$\U (\Z C_8\rtimes C_2)$ up to finite index}
Let $G$ be the group with presentation $\GEN{a,b\mid a^8=1,b^2=1,bab=a^5}$. This is the group with \textsc{SmallGroup} ID [16,6]. The computation of $\U(\Z G)$ is also done in \cite{2005delRio} using Poincar\'e's Polyhedron Theorem. We approach this example with more elementary techniques.

Using Wedderga \cite{Wedderga} based on the papers \cite{Olivieri2004,2009OlteanudelRio}, we compute the Wedderburn decomposition $$\Q G= 4\Q \oplus 2\Q(i) \oplus M_2(\Q(i))$$ and see that the primitive central idempotent $e$ yielding the last simple component is afforded by the pair $(\GEN{a},1)$ and equals $e=e(G,\GEN{a},1)=\frac{1}{2}-\frac{1}{2}a^4$.
Hence the simple component can be described as the crossed product $\Q Ge=\Q\GEN{a}\varepsilon(\GEN{a},1)*\GEN{b}$. It is easy to see that $\suma{b}$ is a non-trivial idempotent, which affords a description of $\Q Ge$ as $M_2(\suma{b}\Q Ge\suma{b})$. Another simple calculation shows that $(\suma{b}a^2e\suma{b})^2=-e\suma{b}$ and hence $\suma{b}a^2e\suma{b}$ is mapped to $\mat{i}{0}{0}{0}$ by an isomorphism between $\Q Ge$ and $M_2(\Q(i))$.

We will determine an explicit isomorphism $$\iota: M_2(\Q(i)) \stackrel{\sim}{\longrightarrow}  \Q Ge \subset \Q G.$$ 
It suffices to find images of the following elements:
$$
	E := \begin{pmatrix}
		1 & 0 \\ 0 & 0
	\end{pmatrix},\ 
	I:= \begin{pmatrix}
		i & 0 \\ 0 & 0
	\end{pmatrix},\ 
	A := \begin{pmatrix}
		0 & 1 \\ 0 & 0
	\end{pmatrix},\ 
	B := \begin{pmatrix}
		0 & 0 \\ 1 & 0
	\end{pmatrix}.
$$
We already know possible images for $E$ and $I$:
$$
	\iota(E) = e\suma{b} \in \frac{1}{4}\Z G, \ \iota(I) = \suma{b}a^2e\suma{b}\in \frac{1}{4}\Z G.
$$
Hence it suffices to find images for $A$ and $B$. These must merely satisfy 
\begin{equation}\iota(A)\cdot \iota(B) = \iota(E),\end{equation}\begin{equation}\label{eqn_1}\iota(E) \cdot \iota(A) \cdot (1-\iota(E)) = \iota(A)\end{equation} and \begin{equation}\label{eqn_2}(1-\iota(E)) \cdot \iota(B) \cdot \iota(E) = \iota(B).\end{equation} Define
$$
	\iota(A) = \suma{b}ae(e-\suma{b})=\suma{b}ae(1-\suma{b})\in \frac{1}{4}\Z G \quad\textrm{and}\quad\iota(B)=\alpha (1-\suma{b})a\inv e\suma{b}\in \alpha\frac{1}{4}\Z G,
$$
where $\alpha \in \Q$ is an element yet to be determined. Equations \eqref{eqn_1} and \eqref{eqn_2} are satisfied by definition. Consider the standard involution
$$
	-^\circ: \Q G \longrightarrow \Q G:\ \sum_{g\in G } a_g g \mapsto \sum_{g\in G } a_g g^{-1} .
$$
Clearly $\iota(B)^\circ = \alpha \iota(A)$. Hence $(\iota(A) \iota(B))^\circ = \iota(A) \iota(B)$. Note that 
$\iota(E)^\circ = \iota(E)$ and $\iota(I)^\circ = -\iota(I)$. It follows that $-^\circ$ induces the standard involution on
the quaternion algebra $\iota(E)  \Q G  \iota(E)$. So the invariance of $\iota(A) \iota(B)$ implies that 
$\iota(A) \iota(B) = q  \iota(E)$ for some $q \in \Q$. We want $q$ to be equal to $1$ and compute
$$
	\iota(A)\iota(B)=\alpha e\suma{b}.
$$
Hence
$$
	\alpha = 1.
$$ 

Now we apply our results to compute $\U(\Z G)$ up to finite index. Since all components but one yield a finite group of units, it suffices to compute the group of units of $1-e+\Z Ge$ up to finite index.

Let $\Z[i]$ be the maximal order of $\Q(i)$, then $\Z Ge$ and $\SL_2(\Z[i])$ are commensurable. From Method \ref{method}, we conclude that $$ 1+\iota \mat{0}{1}{0}{0}, 1+\iota\mat{0}{0}{1}{0}, 1+\iota\mat{0}{i}{0}{0}, 1+\iota\mat{0}{0}{i}{0}$$ generates $\U(1-e+\Z Ge)$ up to finite index. Hence this set of elements in $\Q G$ generates a group which is commensurable with $\U(\Z G)$.

However these images do not lie in $\Z G$ but in $\frac{1}{4}\Z G$. If one is not satisfied with commensurability, one has to deduce from the generators of $\SL_2(\Z[i])$ a list of generators of the congruence subgroup 
$$
C=\begin{pmatrix} 2\Z\left[i\right]+1 & 4\Z\left[i\right] \\ 4\Z\left[i\right] & 2\Z\left[i\right]+1 \end{pmatrix}_{\det=1}
$$
 using techniques as for example Schreier's lemma. 
In order to use Schreier's lemma, one should go down from $\SL_2(\Z[i])$ to $C$ in the lattice of subgroups in small steps.

Jespers and Leal \cite{JespersLeal1991} already deduced generators of $$\overline{\Gamma}=\begin{pmatrix} 2\Z\left[i\right]+1 & 4\Z\left[i\right] \\ 2\Z\left[i\right] & 2\Z\left[i\right]+1 \end{pmatrix}_{\det=1} \bigg / \{\pm \id\}$$ from generators of $\PSL_2(\Z[i])$ using the Reidemeister-Schreier theorem.  They obtained the following list of generators:

$$l_1 = \mat{5-2i}{4+4i}{2}{1+2i},\, l_2 = \mat{5}{4i}{4i}{-3},\, l_3 = \mat{-3-6i}{-4i}{-4i}{1-2i},$$
$$l_4 = \mat{3- 2i}{4}{2i}{-1+ 2i}, \, l_5 = \mat{29+ 2}{44i}{-8 +20i}{ -31 -10i}, \, l_6 =\mat{17 - 12i}{20 + 24i}{ 8}{1 + 12i},$$
$$l_7 = \mat{11 + 38i}{-56 + 20i}{ 22i }{-33 +2i}, \, l_8 = \mat{27 - 2i}{-32 + 16i}{18i}{ -9 - 22i},$$ 
$$l_9 = \mat{-7 - 14i}{20 - 12i}{10 - 8i}{13+ 14i}, \, r_1 = \mat{1}{4}{0}{1}, \, r_2 = \mat{3}{-4i}{2i}{3} , \, r_3 = \mat{1}{0}{2}{1},$$
$$ r_4 = \mat{-7 - 6i}{-4 - 8i}{4 - 4i}{5 - 2i}, \, r_5 = \mat{29 + 44i }{-64 + 48i}{-30 + 16i}{-27 - 44i},$$
$$r_6 =\mat{-5 - 8i}{12 - 8i}{-2i}{3}, \, r_7 = \mat{13 + 6i}{-8 + 20i}{ 12 }{1 + 18i},$$
$$r_8 = \mat{-7 - 18i}{16 + 16i}{12 - 4i}{ -11 +10i}.$$

Clearly this list is also a list of generators of a subgroup of finite index in $$\Gamma=\begin{pmatrix} 2\Z\left[i\right]+1 & 4\Z\left[i\right] \\ 2\Z\left[i\right] & 2\Z\left[i\right]+1 \end{pmatrix}_{{\det=1}}.$$

It is easy to see that the congruence subgroup $C$ has index $4$ in $\Gamma$, and one can check that a right transversal is given by $T=\lbrace I, r_2,r_3,r_2r_3\rbrace$. Applying Schreier's lemma to $\Gamma$ and $C$ with transversal $T$ and list of generators $X=\{l_1,...,l_9,r_1,...,r_8\}$, we get a set of $66$ generators of $C$ up to finite index of the form
$$\lbrace tx(\overline{tx})^{-1} \mid t \in T, x \in X \rbrace,$$
where $\overline{tx} \in T$ is the representative of the right coset containing the element $tx$. Thus, up to finite index, $\U(\Z G)$ is generated by at most $66$ generators which are of the following form
$$ 1-e+\iota\left(tx(\overline{tx})^{-1}\right),$$
for $t \in T$ and $x \in X$.

\begin{remark}
One should note that also the generators of Proposition \ref{prop_generators} could have been used to compute generators up to finite index of a congruence subgroup in $\SL_2(\Z[i])$. 
\end{remark}

\subsection{$\U(\Z\SL(2,5))$ up to commensurability}
By a classification of Frobenius complements \cite[Theorem 18.6]{1968Passman}, we know that $\SL(2,5)$ is the smallest non-solvable Frobenius complement. Since this group does not satisfy the properties needed for the Jespers-Leal-result, we will apply our results to investigate the units of $\Z\SL(2,5)$.
Using \cite{Wedderga}, we compute that $\Q \SL(2,5)$ is isomorphic to
$$ \Q \oplus M_4(\Q) \oplus D_1  \oplus M_2(D_2) \oplus M_5(\Q)  \oplus M_3(D_3) \oplus M_3(\Q[\sqrt{5}])$$ with $$D_1 = \quat{-1,-1}{\Q[\sqrt{5}]}  \quad D_2 = \quat{-1,-3}{\Q} \quad D_3 = \quat{-1,-1}{\Q}.$$ 

The group $\SL(2,5)$ satisfies the condition of Method \ref{method} since it can be shown that the only component where no $\suma{g}$ projects to a non-central idempotent is $\quat{-1,-1}{\Q[\sqrt{5}]}$ \cite[Theorem 9.1]{2001Brown}. 

We take the maximal order $\O=\Z\left[1,i,\frac{1}{2}(1+j),\frac{1}{2}(i+k)\right]$ in the quaternion algebra $D_2$. 
Let $e$ be the primitive central idempotent associated to the component $M_2\left(D_2\right)$ and let $\iota$ be the isomorphism between $M_2\left(D_2\right)$ and $\Q Ge$. Let $U$ be the subset of $\Q G$ containing the elements $$ 1+\iota \mat{0}{1}{0}{0}, 1+\iota\mat{0}{0}{1}{0}, 1+\iota\mat{0}{i}{0}{0}, 1+\iota\mat{0}{0}{i}{0},$$ $$1+\iota \mat{0}{\frac{1}{2}(1+j)}{0}{0}, 1+\iota\mat{0}{0}{\frac{1}{2}(1+j)}{0},$$ $${1+\iota\mat{0}{\frac{1}{2}(i+k)}{0}{0}, 1+\iota\mat{0}{0}{\frac{1}{2}(i+k)}{0}.}$$

Then $\GEN{U \cup \BA \cup \BI}$ is commensurable with $\U(\Z G)$.

\renewcommand{\bibname}{References}
\bibliographystyle{elsarticle-num-names}
\bibliography{references}

\end{document}